\documentclass{amsart}
 \usepackage{graphicx}

\makeatletter
 \def\LaTeX{\leavevmode L\raise.42ex
   \hbox{\kern-.3em\size{\sf@size}{0pt}\selectfont A}\kern-.15em\TeX}
\makeatother

\newcommand{\BibTeX}{{\rm B\kern-.05em{\sc
i\kern-.025emb}\kern-.08em\TeX}}
\newtheorem{col}{Corollary}[section]
\newtheorem{thm}{Theorem}[section]
\newtheorem{lem}[thm]{Lemma}
\newtheorem{prop}[thm]{Proposition}
\newtheorem{rem}[thm]{Remark}
\theoremstyle{defn}

\numberwithin{equation}{section}

\begin{document}

\title[Discrete Hilbert and Kak-Hilbert transforms ]{Sampling and  interpolation for the discrete Hilbert and Kak-Hilbert  transforms}

\author{Isaac Z. Pesenson}

\address{Department of Mathematics, Temple University,
 Philadelphia,
PA,  19122 USA}
\email{pesenson@temple.edu}


\maketitle
\begin{abstract}

 The goal of  the  paper is 
to obtain  analogs of the sampling theorems and  of the 
 Riesz-Boas interpolation formulas  which are relevant to  the Discrete Hilbert  and  Kak-Hilbert  transforms in $l^{2}$.

\end{abstract}

\section{Introduction} 
 The objective of the paper is to establish some analogs of the classical (Shannon) sampling theorems and Riesz-Boas interpolation formulas which are associated with the Discrete Hilbert  transform and  Kak-Hilbert  transform  in $l^{2}$. The basic idea  is to utilize  their one-parameter uniformly bounded groups of  operators  in the space $l^{2}$ to reduce questions about sampling and interpolation to the classical ones. Such approach to sampling and interpolation for general $C_{0}$ one-parameter uniformly bounded groups of operators in Banach spaces was developed in  \cite{Pes15a}, \cite{Pes15b}.
 The main part of the present paper is devoted to the discrete Hilbert transform. Here we show in all the details how one can use one-parameter group of isometries generated by the discrete Hilbert transform in $l^{2}$  to obtain several relevant sampling and interpolation results.

If $\widetilde{H}$ is the discrete Hilbert transform in the space $l^{2}$  with the natural inner product $\langle \cdot, \cdot \rangle$, (see section 2 for all the definitions)  then the bounded  operator $H=\pi \widetilde{H}$ generates a one-parameter group $e^{t H},\>\>t\in \mathbb{R},$ of isometries of $l^{2}$. The fact that $e^{t H},\>\>t\in \mathbb{R},$ is a group of isometries and  the explicit formula for all  $e^{t H}$ were given in  \cite{DC}. In our first sampling Theorem \ref{FST}   we give an explicit formula for a function $\langle e^{t H}{\bf a}, {\bf a}^{*}\rangle,\>\>t\in \mathbb{R},$ for every ${\bf a}, {\bf a}^{*}\in l^{2}$, in terms of equally spaced "samples" $\langle e^{\gamma k H}{\bf a}, {\bf a}^{*}\rangle,\>\>k\in \mathbb{Z},$ for any $0<\gamma<1$. In two other sampling Theorems \ref{SST} and \ref{V-T}  we express the entire trajectory  $e^{t H}{\bf a},\>\>t\in \mathbb{R},\>{\bf a}\in l^{2},$ in terms of the integer translations  $e^{k H}{\bf a},\>\>k\in \mathbb{Z}.$ In section \ref{irreg} we have an analog of a sampling theorem with irregularly spaced "samples".

In section \ref{RB} we present some analogs of the classical Riesz-Boas interpolation formulas. Namely, we give explicit formulas for  $H^{2m-1}{\bf a},\>m\in \mathbb{N},\>{\bf a}\in l^{2},$ in terms of  the vectors  $e^{(k-1/2)H}{\bf a},\>\>k\in \mathbb{Z}$, and for $H^{2m}{\bf a},\>m\in \mathbb{N},$ in terms of $e^{kH}{\bf a},\>\>k\in \mathbb{Z}$.

 In the last section \ref{KH} we briefly describe how similar results can be obtained in the case of the Kak-Hilbert transform.

\section{Some harmonic analysis associated with the discrete Hilbert transform}

We will be interested in the operator $H=\pi \widetilde{H}$ where $\widetilde{H}$ is 
 the discrete Hilbert transform operator
 $$
\widetilde{H}: l^{2}\mapsto l^{2},\>\>\>\widetilde{H}{\bf a}={\bf b},\>\>\> {\bf a}=\{a_{j}\}\in l^{2},\>\>\>{\bf b}=\{b_{m}\}\in l^{2},
$$
which is defined by the formula
\begin{equation}
b_{m}=\frac{1}{\pi}\sum_{n\neq m,\>n\in \mathbb{Z}}\frac{a_{n}}{m-n},\>\>\>m\in \mathbb{Z}.
\end{equation}
Since $H$ is a bounded operator the following exponential series converges in $l^{2}$ for every ${\bf a}\in l^{2}$ and every $t\in \mathbb{R}$
\begin{equation}
e^{tH}{\bf a}=\sum_{k=0}^{\infty}\frac{H^{k}{\bf a}}{k!}t^{k}.
\end{equation}
In fact, $H$ is a generator of a one-parameter group of operators  $e^{tH}, \>t\in \mathbb{R}$, which means that \cite{BB}, \cite{K}
\begin{enumerate}
\item 
$$e^{t_{1}H}e^{t_{2}H}=e^{(t_{1}+t_{2})H}, \>\>\>e^{0}=I,
$$
\item 
$$
e^{-t H}=\left( e^{tH}\right)^{-1},
$$

\item  for every ${\bf a}\in l^{2}$ 
$$ 
\lim_{t\rightarrow 0} \frac{e^{tH}{\bf a}-{\bf a}}{t}=H{\bf a}.
$$
\end{enumerate}

\bigskip

It is clear that for  a general bounded operator $A$ the exponent can be extended to the entire complex plane $\mathbb{C}$ and one has the estimate 
\begin{equation}\label{bound}
\|e^{zA}\|\leq \sum_{k=0}^{\infty}\frac{\|A\|^{k}|z|^{k}}{k!}= e^{\|A\||z|},\>\>\>z\in \mathbb{C}.
\end{equation}
In the nice paper by Laura De Carli and Gohin Shaikh Samad \cite{DC} about the group $e^{tH}$ the following results were obtained (among other interesting results):
\begin{enumerate}
\item the explicit formulas for the operators $e^{tH}$ were given,

\item it was shown that every operator $e^{tH}$ is an isometry in $l^{2}$.

\end{enumerate}
The explicit formulas are given in the next statement.
\begin{thm}The operator $H$ generates in $l^{2}$ a one-parameter group of isometries $ e^{tH}{\bf a}={\bf b},\>\>\>{\bf a}=(a_{n})\in l^{2}, \>\>\>{\bf b}=(b_{m})\in l^{2}, $ which is given by the formulas
$$
b_{m}=\frac{\sin(\pi t)}{\pi}\sum_{n\in \mathbb{Z}}\frac{a_{n}}{m-n+t},\>\>\>
$$
if $t\in \mathbb{R}\setminus \mathbb{Z}$, and 
$$
b_{m}=(-1)^{t}a_{m+t}, \>\>
$$
if $t\in \mathbb{Z}$.

\end{thm}

As it was proved  by Schur \cite{Schur} the operator norm of $\widetilde{H}: l^{2}\mapsto l^{2}$ is one and therefore the operator norm of $H$ is $\pi$. 
It was shown in \cite{Graf}  that although the norm of the operator $H$ is $\pi$, only a strong inequality $\|H{\bf a}\|<\pi\|{\bf a}\|$ can hold
for every non-trivial ${\bf a}\in l^{2}$. 
\bigskip

Let's remind that a  Bernstein class \cite{Akh}, \cite{Nik}, which is denoted as ${\bf B}_{\sigma}^{p}(\mathbb{R}),\>\> \sigma\geq 0, \>\>1\leq p\leq \infty,$ is a linear space of all functions $f:\mathbb{R} \mapsto \mathbb{C}$ which belong to $L^{p}(\mathbb{R})$ and admit  extension to $\mathbb{C}$ as entire functions of exponential type $\sigma$.
A function $f$ belongs to ${\bf B}_{\sigma}^{p}(\mathbb{R})$ if and only if the following Bernstein inequality holds 
$$
\|f^{(k)}\|_{L^{p}(\mathbb{R})}\leq \sigma^{k}\|f\|_{L^{p}(\mathbb{R})}
$$
for all natural $k$. Using the distributional Fourier transform 
$$
\widehat{f}(\xi)=\frac{1}{\sqrt{2\pi}}  \int_{\mathbb{R}} f(x)e^{-i\xi x}dx, \>\>\>f\in L^{p}(\mathbb{R}), \>\>1\leq p\leq \infty,
$$
one can show (Paley-Wiener theorem) that $f\in {\bf B}_{\sigma}^{p}(\mathbb{R}), \>\>1\leq p\leq \infty,$ if and only if $f\in L^{p}(\mathbb{R}),\>\>1\leq p\leq \infty,$ and the support of $\widehat{f}$ (in sens of distributions) is in $[-\sigma, \sigma]$.

In what follows the notation $\|\cdot\|$ will always mean $\|\cdot\|_{l^{2}}$.
We note that since $H$ is a bounded operator whose norm is $\pi$ one has for all ${\bf a}\in l^{2}$  the following Bernstein-type  inequality 
\begin{equation}\label{bern-1}
\|H^{k}{\bf a}\|\leq \pi^{k}\|{\bf a }\|.
\end{equation}

Pick an ${\bf a}^{*}\in l^{2}$ and  consider a scalar-valued function 
$$
\Phi(t)=\langle e^{t H}{\bf a}, {\bf a}^{*}\rangle,\>\>\>\>t\in \mathbb{R}.
$$
The following Lemma and the  Corollary   after it can be considered as analogs of the Paley-Wiener theorem.
\begin{lem}\label{LEM}
For every ${\bf a}\in l^{2}$ and every ${\bf a}^{*}\in l^{2}$ the function $\Phi$ belongs to the Bernstein class ${\bf B}_{\pi}^{\infty}(\mathbb{R}).$ 

\end{lem}
\begin{proof}

We  notice that
$$
\left(\frac{{d}}{dt}\right)^{k}\Phi|_{t=0}=\langle H^{k}{\bf a}, {\bf a}^{*}\rangle.
$$
Since the operator norm of $H$ is $\pi$  we obtain that the Taylor series for $\Phi$ converges absolutely on $\mathbb{C}$:
\begin{equation}\label{type-pi}
\left |\sum_{k=0}\left(\frac{{d}}{dt}\right)^{k}\Phi|_{t=0}\frac{z^{k}}{k!}\right|\leq 
\sum_{k=0}\left|\langle H^{k}{\bf a}, {\bf a}^{*}\rangle\right|\frac{|z|^{k}}{k!}= \|{\bf a}\| \|{\bf a^{*}}\|\>e^{\pi |z|},
\end{equation}
and represents there a function of the exponential type $\pi$. 
In addition, the function $\Phi$ is bounded on the real line
$$
\left|\Phi(t)\right|=\left|\langle e^{tH}{\bf a}, {\bf a}^{*}\rangle\right|\leq \|{\bf a}\|\|{\bf a}^{*}\|.
$$
 Lemma is proven.

\end{proof}

This Lemma can also be reformulated as follows.
\begin{col}\label{COL}
For a fixed ${\bf a}\in l^{2}$ the vector-valued function
\begin{equation}
e^{tH}{\bf a}: \mathbb{R}\mapsto l^{2}
\end{equation}
has extension $ e^{z H}{\bf a},\>\>z\in \mathbb{C},$ to the complex plain as an entire function of the exponential type $\pi$ which is bounded on the real line.
\end{col}
We already observed that the function  $\Phi$ for any ${\bf a},\>{\bf a}^{*}\in l^{2}$ belongs to  ${\bf B}_{\pi}^{\infty}(\mathbb{R})$. Let us introduce a new function defined by the next formula if $t\neq 0$ 
\begin{equation}\label{F1}
\Psi(t)=\frac{\Phi(t)-\Phi(0)}{t}=\left\langle \frac{e^{tH}{\bf a}-{\bf a}}{t}, {\bf a}^{*}\right\rangle,
\end{equation}
and in the case $t=0$ by the formula
\begin{equation}\label{F2}
\Psi(0)=\frac{d}{dt}\Phi(t)|_{t=0}=\langle H{\bf a}, {\bf a}^{*}\rangle.
\end{equation}
\begin{lem}
For every ${\bf a}\in l^{2},\>{\bf a}^{*}\in l^{2}$ the function $\Psi$ is in the Bernstein class ${\bf B}_{\sigma}^{2}(\mathbb{R})$.
\end{lem}
\begin{proof}
The function $\Psi$ is an entire function of the  exponential type $\pi$. Indeed, the fact that  
$
\Phi(t)
$ 
is in ${\bf B}_{\pi}^{\infty}(\mathbb{R})$ means \cite{Nik} that
$$
\langle e^{tH}{\bf a},{\bf a}^{*}\rangle=\langle {\bf a},{\bf a}^{*}\rangle+\sum_{k=1}^{\infty} c_{k}t^{k},
$$
with $\>\>\>\overline{\lim}_{k\rightarrow \infty}\sqrt[k]{k!|c_{k}|}\leq \pi $, and then
$$
\Psi(t)=\frac{\langle e^{tH}{\bf a},{\bf a}^{*}\rangle -\langle {\bf a},{\bf a}^{*}\rangle}{t}=\sum_{k=1}^{\infty} c_{k}t^{k-1},
$$
where one obviously has  $\>\>\>\overline{\lim}_{k\rightarrow \infty}\sqrt[k]{k!|c_{k+1}|}\leq \pi $.
In addition, $\Psi$ belongs to $L^{2}(\mathbb{R})$  since  according to the Schwartz  inequality
$$
\left|\Psi(t)\right|^{2}=\left|\left\langle \frac{e^{tH}{\bf a}-{\bf a}}{t}, {\bf a}^{*}\right\rangle\right|^{2}\leq \frac{\left(2\|{\bf a}\|\|{\bf a}^{*}\|\right)^{2}}{|t|^{2}},\>\>\>t\geq 1.
$$
 In other words, $\Psi$ is in the Bernstein class ${\bf B}_{\sigma}^{2}(\mathbb{R})$.  Lemma is proved.
 \end{proof}
 
 We also have the following version of the so-called General Parseval formula \cite{BFHSS}.
  \begin{thm}
 For every ${\bf a},\>{\bf a}^{*}, {\bf b},\>{\bf b}^{*}\in l^{2}$ the next equality holds 
 $$
 \int_{\mathbb{R}}\left\langle \frac{e^{tH}{\bf a}-{\bf a}}{t}, {\bf a}^{*}\right\rangle  \left\langle  \frac{e^{t H}{\bf b}-{\bf b}}{t}, {\bf b}^{*}\right\rangle\>dt=
 $$
 $$
  \left\langle H{\bf a}, {\bf a}^{*}\right\rangle \left\langle H{\bf b}, {\bf b}^{*}\right\rangle +\sum_{k\neq 0} \left\langle \frac{e^{kH}{\bf a}-{\bf a}}{k}, {\bf a}^{*}\right\rangle   \left\langle  \frac{e^{kH}{\bf b}-{\bf b}}{k}, {\bf b}^{*}\right\rangle.
 $$
 \end{thm}

\section{Sampling theorems with regularly spaced samples for orbits $e^{tH}{\bf a}$ }

Below we are going to use the following known fact (see  \cite{BSS}).

\begin{thm}\label{Shannon-1}

If 
 $h\in {\bf B}_{\sigma}^{\infty}(\mathbb{R})$, then for any $0<\gamma<1$ the following formula holds 

\begin{equation}
h(z)=\sum_{k\in \mathbb{Z}} h\left(\gamma\frac{k\pi}{\sigma}\right)\>
sinc\left(\gamma^{-1}\frac{\sigma}{\pi} z-k\right),\>\>\>z\in \mathbb{C},
\end{equation} 
where the series converges uniformly on compact subsets of $\mathbb{C}$.
\end{thm}

By using Theorem \ref{Shannon-1} and Lemma \ref{LEM} we obtain 
our First Sampling Theorem. 
\begin{thm}\label{FST}For  every ${\bf a},\>{\bf a}^{*}\in l^{2}$, every $0<\gamma< 1,$ and every $z\in \mathbb{C}$ one has

\begin{equation}\label{1-2}
\left\langle e^{zH}{\bf a}, {\bf a}^{*}\right\rangle=\sum_{k\in \mathbb{Z}} \left\langle e^{(\gamma k) H }{\bf a}, {\bf a}^{*}\right\rangle\>
sinc\left(\gamma^{-1}z -k\right),
\end{equation}
where the series converges uniformly on compact subsets of $\mathbb{R}$.
\end{thm}

Explicitly, the formula (\ref{1-2}) means that 
if $z$ in (\ref{1-2}) is a real $z=t$ which is not an integer then 
  (\ref{1-2}) takes the form
\begin{equation}
\frac{\sin \pi t}{\pi}\sum_{m\in \mathbb{Z}} \sum_{n\in \mathbb{Z}\>\>n\neq m}   \frac{a_{n}b_{m}} {m-n+t}=S_{1}+S_{2}, 
\end{equation}
where
$$
S_{1}=\sum_{k\in \mathbb{Z},\> \gamma k\in \mathbb{R}\setminus\mathbb{Z}}\frac{\sin \pi \gamma k}{\pi}\left(    \sum_{m\in \mathbb{Z}}\sum_{n\in \mathbb{Z}\>\>n\neq m}  \frac{a_{n}b_{m}}{m-n+\gamma k      } \right)\>sinc(\gamma^{-1}t-k),
$$
and 
$$
S_{2}=\sum_{k\in \mathbb{Z};\>\gamma k\in\mathbb{Z}} \left( \sum_{m\in \mathbb{Z}}(-1)^{\gamma k}a_{m+\gamma k}b_{m}\right)\>
sinc(\gamma^{-1}t-k).
$$

Next, if $z=t$ in (\ref{1-2}) is an integer   then the formula  (\ref{1-2}) is given by 
$$
\sum_{m\in \mathbb{Z}}(-1)^{t}a_{m+t}b_{m}=S_{1}+S_{2}.
$$
We note, that if $t=\gamma N,\>\>N\in \mathbb{Z}$, then (\ref{1-2}) is evident since its both sides are (obviously) identical:
$$
\langle e^{tH}{\bf a}, {\bf a}^{*}\rangle=\langle e^{tH}{\bf a}, {\bf a}^{*}\rangle.
$$

\begin{rem}
The situation with such kind "obvious interpolation"   is very common for sampling formulas. Consider for example 
 the following classical (Shannon) formula 
\begin{equation}
f(t)=\sum_{k\in \mathbb{Z}} f\left(k\right)\>sinc\left(t-k\right),
\end{equation}
for  $f\in {\bf B}_{\pi}^{2}(\mathbb{R}),$ where the series converges uniformly on compact subsets of $\mathbb{R}$ and also in  $L^{2}(\mathbb{R})$. This formula is informative only when $t$ is not integer. When $t=N\in \mathbb{Z},$ it clearly becomes a tautology $f(N)=f(N)$ because $sinc \>z$ is zero for every $z\in \mathbb{Z}\setminus \{0\}$ and $sinc\> 0=1$.
\end{rem}

 We are going to use the 
 next known result (see\cite{BSS})

\begin{thm}\label{Shannon-2}

If $\>h\in {\bf B}_{\sigma}^{2}(\mathbb{R})\>$ then the following formula holds for $z\in \mathbb{C}$ 
\begin{equation}
h(t)=\sum_{k\in \mathbb{Z}} h\left(\frac{k\pi}{\sigma}\right)\>sinc\left(\frac{\sigma}{ \pi}t-k\right),
\end{equation}
where the series converges uniformly on compact subsets of $\mathbb{R}$. The restriction of the series to the real line also converges  in $L^{2}(\mathbb{R})$.
\end{thm}

\begin{thm}\label{SST}For every ${\bf a}\in l^{2}$

\begin{equation}\label{2-2}
e^{tH}{\bf a}=
 {\bf a}+t\>sinc \>\left(t\right) H{\bf  a}\>+
t\sum_{k\in \mathbb{Z}\setminus \{0\}}  \frac{e^{kH}{\bf a}-{\bf a}}{k}
\>
sinc\left(t -k\right),
\end{equation}
where the series converges in the norm of $l^{2}$.

\end{thm}
\begin{proof} Since $\Psi$ is in ${\bf B}_{\sigma}^{2}(\mathbb{R})$ one can use Theorem \ref{Shannon-2} to 
obtain the following formula for every ${\bf a},\>{\bf a}^{*}\in l^{2}$, and every $t\in \mathbb{R}$,    

\begin{equation}\label{F11}
 \left\langle \frac{e^{tH}{\bf a}-{\bf a}}{t}, {\bf a}^{*}\right\rangle
=\sum_{k\in \mathbb{Z}} \left\langle \frac{e^{kH}{\bf a}-{\bf a}}{k}, {\bf a}^{*}\right\rangle
\>
sinc\left(t -k\right),
\end{equation}
where the series converges uniformly on compact subsets of $\mathbb{R}$.  Actually, this formula means that if $t\neq 0$ then 

\begin{equation}\label{F11a}
 \left\langle \frac{e^{tH}{\bf a}-{\bf a}}{t}, {\bf a}^{*}\right\rangle=\langle H{\bf a},{\bf a}^{*}\rangle\>sinc(t)+\sum_{k\in \mathbb{Z}\setminus\{0\}} \left\langle \frac{e^{kH}{\bf a}-{\bf a}}{k}, {\bf a}^{*}\right\rangle
\>
sinc\left(t -k\right),
\end{equation}
and for $t=0$ it becomes just 
$$
\langle H{\bf a},{\bf a}^{*}\rangle=\langle H{\bf a},{\bf a}^{*}\rangle.
$$
The formula  (\ref{F11a}) can be rewritten as 

\begin{equation}\label{F1-5}
\left\langle e^{tH}{\bf a}, {\bf a}^{*}\right\rangle=
$$
$$
\langle {\bf a}, {\bf a}^{*}\rangle+t\left\langle H{\bf a}, {\bf a}^{*}\right\rangle\>sinc \>(t)+
t\sum_{k\in \mathbb{Z}\setminus \{0\}} \left\langle \frac{e^{kH}{\bf a}-{\bf a}}{k}, {\bf a}^{*}\right\rangle
\>
sinc\left(t -k\right).
\end{equation}
Next, we notice that the series 
$$
\sum_{k\in \mathbb{Z}\setminus \{0\}} \frac{e^{kH}{\bf a}-{\bf a}}{k}
\>sinc\left(t -k\right),
$$
converges in $l^{2}$ since for every fixed $t\in \mathbb{R}$:
$$
\left\|\sum_{k\in \mathbb{Z}\setminus \{0\}} \frac{e^{kH}{\bf a}-{\bf a}}{k}
\>sinc\left(t -k\right)
\right\|\leq 2\|{\bf a}\|\sum_{k\neq 0, t} \frac{1}{|k||t-k|}<\infty.
$$
It allows to rewrite (\ref{F1-5}) as
$$
\left\langle e^{tH}{\bf a}, {\bf a}^{*}\right\rangle=
\left\langle {\bf a}+t H{\bf  a}\>sinc \>\left(t\right)+
t\sum_{k\in \mathbb{Z}\setminus \{0\}}  \frac{e^{kH}{\bf a}-{\bf a}}{k}
\>
sinc\left(t -k\right), {\bf a}^{*}\right\rangle.
$$
Since this equality holds for all sequences $\mathbf{a}^{*}\in l^{2}$ we obtain (\ref{2-2}). Theorem is proven.
\end{proof}

Next, we reformulate (\ref{2-2}) in its "native" terms.

\begin{prop}
If ${\bf a }=(a_{n})\in l^{2}$  and $t$ is not integer then the left-hand side of (\ref{2-2}) is a sequence $e^{tH}{\bf a}={\bf b}=(b_{m})\in l^{2}$ 
with the entries 
\begin{equation}\label{Sec-1}
b_{m}=\frac{\sin \pi t}{\pi}\sum_{n\in \mathbb{Z}}   \frac{a_{n}} {m-n+t},
\end{equation}
and the right-hand side represents a sequence ${\bf c}=(c_{m})\in l^{2}$  with the entries
\begin{equation}\label{sec-2}
c_{m}=a_{m}+
t\>sinc(t)\sum_{n\in \mathbb{Z}, n\neq m}\frac{a_{n}}{m-n}+
$$
$$
t\sum_{k\in \mathbb{Z}\setminus \{0\}}  
\frac{  (-1)^{k}a_{m+k}-a_{m}}{k}\>
sinc\left(t -k\right).
\end{equation}
If $t$ in  (\ref{2-2}) is an integer $t=N$  then $b_{m}=(-1)^{N}a_{m+N}$ and 
$$
c_{m}=a_{m}+
N\sum_{n\in \mathbb{Z}, n\neq m}\frac{a_{n}}{m-n}\>sinc(N)+
$$
$$
N\sum_{k\neq 0}  
\frac{  (-1)^{k}a_{m+k}-a_{m}}{k}\>
sinc\left(N -k\right)=(-1)^{N}a_{m+N}.
$$
Thus in the case when $t=N$ is an integer we obtain just a tautology
$$
b_{m}=(-1)^{N}a_{m+N}=c_{m}.
$$
\end{prop}

The next theorem is a generalization of what is known as the Valiron-Tschakaloff
sampling/interpolation formula \cite{BFHSS}.

\begin{thm}\label{V-T}
For every ${\bf a}\in l^{2}$
one  has 
\begin{equation}\label{VT-2}
e^{tH}{\bf a}=sinc\left( t\right){\bf a}+t \>sinc\left( t\right)H{\bf a}+  \sum_{k\in \mathbb{Z}\setminus \{0\}}\frac{ t}{k}sinc\left( t-k\right)e^{kH}{\bf a},
\end{equation}
where the series converges in the norm of $l^{2}$.
\end{thm}

\begin{proof}
If $h\in \mathbf{B}_{\sigma}^{\infty}(\mathbb{R}),\>\>\>\sigma>0,$ then for all $z \in \mathbb{C}$ the following 
Valiron-Tschakaloff sampling/interpolation formula holds \cite{BFHSS} 
\begin{equation}\label{vt}
h(t)= sinc\left(\frac{\sigma  t}{\pi}\right)f(0)+
$$
$$
t \> sinc\left(\frac{\sigma  t}{\pi}\right)f^{'}(0)+
 \sum_{k\in \mathbb{Z}\setminus \{0\}}\frac{\sigma t}{k\pi} sinc\left(\frac{\sigma  t}{\pi}-k\right)h\left(\frac{k\pi}{\sigma}\right),
\end{equation}
the convergence being absolute and uniform on compact subsets of $\mathbb{C}$.
If  ${\bf a}, {\bf a}^{*}\in l^{2}, $  then  $\left\langle e^{tH}{\bf a},\>{\bf a}^{*}\right\rangle\in  \mathbf{B}_{\pi}^{\infty}(\mathbb{R})$ and according to (\ref{vt}) with $\sigma=\pi$ we have
$$
\left\langle e^{tH}{\bf a},\>{\bf a}^{*}\right\rangle=
sinc\left( t\right)\left\langle {\bf a},\>{\bf a}^{*}\right\rangle+
$$
$$
t \> sinc\left( t\right)\left\langle H{\bf a},\>{\bf a}^{*}\right\rangle+
 \sum_{k\in \mathbb{Z}\setminus \{0\}}\frac{t}{k} sinc\left( t-k\right)\left\langle e^{kH}{\bf a},\>{\bf a}^{*}\right\rangle.
$$
Because  the series 
$$
\sum_{k\in \mathbb{Z}\setminus \{0\}}\frac{ t}{k} sinc\left( t-k\right)e^{kH}{\bf a}
$$
converges in $l^{2}$ we obtain the formula (\ref{VT-2}). Theorem is proved.

\end{proof}
The following proposition formulates (\ref{VT-2}) in the specific language of $l^{2}$.

\begin{prop}
If ${\bf a }=(a_{n})\in l^{2}$  and $t$ is not integer then the left-hand side of (\ref{VT-2}) is a sequence $e^{tH}{\bf a}={\bf b}=(b_{m})\in l^{2}$ 
with entries 

\begin{equation}\label{Sec-1}
b_{m}=\frac{\sin \pi t}{\pi}\sum_{n\in \mathbb{Z}}   \frac{a_{n}} {m-n+t},
\end{equation}
and the right-hand side of (\ref{VT-2}) represents a sequence ${\bf c}=(c_{m})\in l^{2}$  with entries
\begin{equation}\label{sec-2}
c_{m}=a_{m}\>sinc(t)+
$$
$$
t\>sinc(t) 
\sum_{n\in \mathbb{Z}, n\neq m}\frac{a_{n}}{m-n}+
t\sum_{k\in \mathbb{Z}\setminus \{0\}} (-1)^{k}\frac{ sinc\left(t -k\right) }{k  } a_{m+k}.
\end{equation}
\end{prop}
When $t$ in  (\ref{VT-2}) is an integer $t=N$  then (\ref{VT-2}) is the tautology $b_{m}=(-1)^{N}a_{m+N}=c_{m}$.

\section{An irregular sampling theorem}\label{irreg}

The following fact was proved in \cite{Hig}.

\begin{thm}\label{Hig}
Let $\{t_{k}\}_{k\in \mathbb{Z}}$ be a sequence of real numbers such that 
\begin{equation}\label{Hig1}
\sup_{k\in\mathbb{Z}}|t_{k}-k|<1/4.
\end{equation}
Define the entire function 
\begin{equation}\label{Hig2}
G(t)=(t-t_{0})\prod_{k\in \mathbb{Z}}\left(1-\frac{z}{t_{k}}\right)\left(1-\frac{z}{t_{-k}}\right).
\end{equation}
 Then for all $f\in \mathbf{B}_{\pi}^{2}(\mathbb{R})$ we have
$$
f(t)=\sum_{k\in \mathbb{Z}}f(t_{k})\frac{G(t)}{G^{'}(t_{k})(t-t_{k})}
$$
uniformly on every compact subset of $\mathbb{R}$.
\end{thm}

As it was already noticed, for any ${\bf a},{\bf a}^{*}\in l^{2}$ the function $\Psi(t)$ defined for all $t\neq 0$ as 

$$
\left\langle \frac{e^{tH}{\bf a}-{\bf a}}{t}, {\bf a}^{*}\right\rangle,
$$
and for $t=0$ as $\Psi(0)=\langle {H\bf a}, {\bf a}^{*}\rangle$ belongs to $\mathbf{B}_{\pi}^{2}(\mathbb{R})$.
 Applying Theorem \ref{Hig} we obtain the following.

 \begin{thm} If ${\bf a},{\bf a}^{*}\in l^{2}$ and a sequence $\{t_{k}\}$ satisfies (\ref{Hig1}) then 
 $$
\Psi(t)=\sum_{k\in \mathbb{Z}}\Psi(t_{k})\frac{G(t)}{G^{'}(t_{k})(t-t_{k})},
$$
uniformly on every compact subset of $\mathbb{R}$.
\end{thm}

\section{Riesz-Boas interpolation formulas for the discrete Hilbert transform}\label{RB}

Consider a trigonometric
polynomial $P(t)$ of one variable $t$ as a function on a
unit circle $\mathbb{T}$. The famous Riesz interpolation formula \cite{R1}, \cite{R}, \cite{Nik}  can be written in the
form
\begin{equation}\label{Riesz}
\left(\frac{d}{dt}\right)P(t)=\frac{1}{4\pi}\sum_{k=1}^{2n}(-1)^{k+1}
\frac{1}{\sin^{2}\frac{t_{k}}{2}}U_{t_{k}}P(t), \>\>\>t\in
\mathbb{T}, \>\>\>t_{k}=\frac{2k-1}{2n}\pi,
\end{equation}
where $U_{t_{k}}P(t)=P(t_{k}+t)$.
This formula was extended by Boas \cite{B}, \cite{B1}, (see also \cite{Akh}, \cite{Nik}, \cite{Schm}) to functions in $\mathbf{B}_{\sigma}^{\infty}(\mathbb{R})$ in the following form 
\begin{equation}\label{Boas}
\left( \frac{d}{dt}\right)f(t)=\frac{n}{\pi^{2}}\sum_{k\in\mathbb{Z}}\frac{(-1)^{k-1}}{(k-1/2)^{2}}
U_{\frac{\pi}{n}(k-1/2)}f(t), \>\>\>t\in \mathbb{R},
\end{equation}
where $U_{\frac{\pi}{n}(k-1/2)}f(t)=f(\frac{\pi}{n}(k-1/2)+t)$.
In turn, the formula (\ref{Boas}) was  extended in \cite{BSS} to  higher powers   $(d/dt)^{m}$. In this section  we present some natural analogs of such formulas (which we call Riesz-Boas interpolation formulas) associated with the discrete Hilbert transform. Our objective   is to obtain  similar formulas where the operator $d/dt$ is replaced by the discrete Hilbert transform $H$ and the group of regular translations $U_{t}$  is replaced by the group $e^{t H}$.

Let's introduce bounded operators
\begin{equation}\label{b1}
\mathcal{R}^{(2s-1)}_{H}{\bf a}=
\sum_{k\in \mathbb{Z}}(-1)^{k+1}A_{s,k}
e^{(k-1/2)H}{\bf a}
,\>\> \>\>{\bf a}\in l^{2},\>\>\>s\in \mathbb{N},
\end{equation}
and
\begin{equation}\label{b2}
\mathcal{R}^{(2s)}_{H}{\bf a}=
\sum_{k\in \mathbb{Z}}(-1)^{k+1}B_{s,k}
e^{kH}
{\bf a}, \>\> \>\>{\bf a}\in l^{2}, \>\>s\in \mathbb{N},
\end{equation}
where $A_{s,k}$ and $B_{s,k}$ are defined 
as 

\begin{equation}\label{A}
A_{s,k}=(-1)^{k+1}  sinc ^{(2s-1)}\left(\frac{1}{2}-k\right)=
$$
$$
\frac{(2s-1)!}{\pi(k-\frac{1}{2})^{2s}}\sum_{j=0}^{s-1}\frac{(-1)^{j}}{(2j)!}\left(\pi\left(k-\frac{1}{2}\right)\right)^{2j},\>\>\>s\in \mathbb{N},
\end{equation}
for $k\in \mathbb{Z}$,

\begin{equation}\label{B}
B_{s,k}=(-1)^{k+1}  sinc ^{(2s)}(-k)=\frac{(2s)!}{\pi k^{2s+1}}\sum_{j=0}^{s-1}\frac{(-1)^{j}(\pi k)^{2j+1}}{(2j+1)!},\>\>\>s\in \mathbb{N},
\end{equation}
for $k\in \mathbb{Z}\setminus \{0\}$, 
and 
\begin{equation}\label{B0}
B_{s,0}=(-1)^{s+1} \frac{\pi^{2s}}{2s+1},\>\>\>s\in \mathbb{N}.
\end{equation}
Both series converge in $l^{2}$ due to the following formulas (see \cite{BSS})
\begin{equation}
\sum_{k\in \mathbb{Z}}\left|A_{s,k}\right|=\pi^{2s-1},\>\>\>\>\>
\sum_{k\in \mathbb{Z}}\left|B_{s,k}\right|=\pi^{2s}\label{id-2}, s\in \mathbb{N}.
 \end{equation}
 Since $\|e^{tH}f\|=\|f\|$ it implies that 
  \begin{equation}\label{norms}
 \|\mathcal{R}^{(2s-1)}_{H}{\bf a}\|\leq \pi^{2s-1}\|{\bf a}\|,\>\>\>\>\>\|\mathcal{R}^{(2s)}_{H}{\bf a}\|\leq \pi^{2s}\|{\bf a}\|,\>\>\>{\bf a}\in l^{2},\>\>\>s\in \mathbb{N}. 
 \end{equation}
 \begin{thm}
 For ${\bf a}\in l^{2}$  the following Riesz-Boas-type interpolation formulas hold true for $r\in \mathbb{N}$
\begin{equation}\label{B1}
H^{r}{\bf a}=\mathcal{R}^{(r)}_{H}{\bf a},\>\>\>{\bf a}\in l^{2}.
\end{equation}
More explicitly, if $r=2s-1, \>s\in \mathbb{N},$ then 
\begin{equation}\label{RB1}
H^{2s-1}{\bf a}=\sum_{k\in \mathbb{Z}}(-1)^{k+1}A_{s,k}
e^{(k-1/2)H}{\bf a},
\end{equation}
and when $r=2s,\>s\in \mathbb{N},$ then 
\begin{equation}\label{RB2}
H^{2s}{\bf a}=
\sum_{k\in \mathbb{Z}}(-1)^{k+1}B_{s,k}
e^{kH}
{\bf a}.
\end{equation}
\end{thm}

\begin{proof}
As we know, for any ${\bf a},\>{\bf a}^{*}\in l^{2}$ the function 
 $
\Phi(t)=\left\langle e^{tH}{\bf a}, {\bf a}^{*}\right\rangle
$
belongs to $ {\bf B}_{\pi}^{\infty}(\mathbb{R}).$ Thus by  \cite{BSS} we have
$$
\Phi^{(2m-1)}(t)=\sum_{k\in \mathbb{Z}}(-1)^{k+1}A_{m,k}\Phi\left(t+(k-1/2)    \right),\>\>\>m\in \mathbb{N},
$$
$$
\Phi^{(2m)}(t)=\sum_{k\in \mathbb{Z}}(-1)^{k+1}B_{m,k}\Phi\left(t+k \right),\>\>\>m\in \mathbb{N}.
$$
Together with 
$$
\left(\frac{d}{dt}\right)^{k}\Phi(t)=\left<H^{k}e^{tH}{\bf a}, {\bf a}^{*}\right>,
$$
  it shows
$$
\left<e^{tH}H^{2m-1}{\bf a}, {\bf a}^{*}\right>=\sum_{k\in \mathbb{Z}}(-1)^{k+1}A_{m,k}\left<e^{\left(t+(k-1/2)    \right)H}{\bf a},\>\>{\bf a}^{*}\right>,\>\>\>m\in \mathbb{N},
$$
and also
$$
\left<e^{tH}H^{2m}{\bf a}, {\bf a}^{*}\right>=\sum_{k\in \mathbb{Z}}(-1)^{k+1}B_{m,k}\left<e^{\left(t+k   \right)H}{\bf a},\>\> {\bf a}^{*}\right>,\>\>\>m\in \mathbb{N}.
$$
Since both series (\ref{b1}) and (\ref{b2}) converge in $l^{2}$ and the last two equalities hold for any ${\bf a}^{*}\in l^{2}$ we obtain the next two formulas
\begin{equation}\label{sam1}
e^{tH}H^{2m-1}{\bf a}=\sum_{k\in \mathbb{Z}}(-1)^{k+1}A_{m,k}e^{\left(t+(k-1/2)   \right)H}{\bf a},\>\>\>m\in \mathbb{N},
\end{equation}
\begin{equation}\label{sam2}
e^{tH}H^{2m}{\bf a}=\sum_{k\in \mathbb{Z}}(-1)^{k+1}B_{m,k}e^{\left(t+k   \right)H}{\bf a}, \>\>\>m\in \mathbb{N}.
\end{equation}
In turn, when $t=0$ these formulas become formulas (\ref{B1}).   
Theorem is proved.
\end{proof}
 Let us introduce the  notation
 $$
  \mathcal{R}_{H}= \mathcal{R}_{H}^{(1)}.
 $$
 One has the following "power" formula which easily follows from the fact that operators $
  \mathcal{R}_{H}$ and $H$ commute.
 \begin{col}
 For any $r\in \mathbb{N}$ and  any ${\bf a}\in l^{2}$
 \begin{equation}\label{powerB}
 H^{r}{\bf a}=\mathcal{R}_{H}^{(r)}{\bf a}=\mathcal{R}_{H}^{r}{\bf a},
 \end{equation}
 where $\mathcal{R}_{H}^{r}{\bf a}=\mathcal{R}_{H}\left(... \left(\mathcal{R}_{H}{\bf a}\right)\right).$
\end{col}

Let's express (\ref{B1}) in terms of $H$ and $e^{tH}$. Our  starting sequence is  ${\bf a}=(a_{n})$ and  then 
we use the notation $H^{k}{\bf a}=\left(a^{(k)}_{m}\right),\>k\in \mathbb{Z}$. One has

$$
H{\bf a}=\frac{1}{\pi}\sum_{(n,\>n\neq n_{1})}\frac{a_{n}}{n_{1}-n}=\left( a^{(1)}_{n_{1}} \right),
$$
$$
H^{2}{\bf a}=
\frac{1}{\pi^{2}}\sum_{(n_{1},\>n_{1}\neq n_{2})}\sum_{(n,\>n\neq n_{1})}\frac{a^{(1)}_{n_{1}}}{(n_{2}-n_{1})(n_{1}-n)}=\left(a^{(2)}_{n_{2}}\right)
$$
and so on up to a $r\in \mathbb{N}$
\begin{equation}\label{gen-term}
H^{r}{\bf a}=
\frac{1}{\pi^{r}}\sum_{(n_{r-1},\>n_{r-1}\neq n_{r})} \sum_{(n_{r-2},\>n_{r-2}\neq n_{r-1})}...
$$
$$
 ...\sum_{(n_{1},\>n_{1}\neq n_{2})} \sum_{(n,\>n\neq n_{1})}\frac{a^{(r-1)}_{n_{r-1}}}{(n_{r}-n_{r-1}) (n_{r}-n_{r-1})... (n_{2}-n_{1})(n_{1}-n)}=\left(a^{(r)}_{n_{r}}\right).
\end{equation}
\begin{thm}
For any ${\bf a}=(a_{n})\in l^{2},\>\>$ and $r=2s-1$ we have in (\ref{RB1}) equality of two sequences where on the left-hand side  we have a sequence whose general term is  
$a^{(2s-1)}_{m}$,
and on the right-hand side we have a sequence whose general term is $c_{m,s}$ where 
$$
c_{m,s}=\sum_{k\in \mathbb{Z}}(-1)^{k+1}
\frac{\sin(\pi (k-1/2))}{\pi} A_{s,k}\sum_{n\neq m}\frac{a_{n}}{m-n+(k-1/2)}.
$$
The equality (\ref{RB1}) tells  that $\>\>a^{(2s-1)}_{m}=c_{m,s}$.

For the case $r=2s$ a sequence on the left hand-side of (\ref{RB2}) has a general term $a^{(2s)}_{m}$, and a sequence on the right has a general term $d_{m,s}$ of the form
$$
d_{m,s}=-\sum_{k\in \mathbb{Z}}B_{s,k}a_{m+k},
$$
and (\ref{RB2}) means that $\>\>a^{(2s)}_{m}=d_{m,s}$.
\end{thm}
Let us introduce the following notations
$$\mathcal{R}^{(2s-1)}_{H}(N){\bf a}=
\sum_{|k|\leq N}(-1)^{k+1}A_{s,k}e^{(k-1/2)H}{\bf a},
$$
$$
\mathcal{R}^{(2s)}_{H}(N){\bf a}=
\sum_{|k|\leq N}(-1)^{k+1}B_{s,k}e^{kH}{\bf a}.
$$
One obviously has the following set of approximate Riesz-Boas-type formulas. 
\begin{thm}
If ${\bf a}\in l^{2}$ and $r\in \mathbb{N}$ then
\begin{equation}\label{appB}
H^{r}{\bf a}=\mathcal{R}_{H}^{(r)}( N){\bf a}+O(N^{-2}).
\end{equation}
\end{thm}
The next Theorem contains another Riesz-Boas-type formula. 
\begin{thm}
If ${\bf a}\in l^{2}$ then the following sampling formula holds for ${\bf a}\in \mathbb{R}$ 
and  $n\in \mathbb{N}$
\begin{equation}\label{s2}
H^{n}e^{tH}{\bf a}= \left(  n\>  sinc^{(n-1)}\left(t\right)  +t\> sinc^{(n)}\left(t\right)
   \right)    H{\bf a}            +
$$
$$
\sum_{k\neq 0}\left(  n\>  sinc^{(n-1)}\left(t-k\right)  +t\> sinc^{(n)}\left(t-k\right)
   \right)\frac{e^{kH}{\bf a}-{\bf a}}{k}.
\end{equation}
where the series converges in the norm of $l^{2}$.
In particular, for $n\in \mathbb{N}$
one has 
\begin{equation}\label{Q}
H^{n}{\bf a}=\mathcal{Q}^{n}_{H}{\bf a},
\end{equation}
where the bounded operator $\mathcal{Q}_{H}^{(n)}$ is given by the formula
\begin{equation}\label{s3}
\mathcal{Q}_{H}^{(n)}f=  \left(   sinc^{(n-1)}\left(0\right)  + sinc^{(n)}\left(0\right)
   \right) n \>H{\bf a}    +
$$
$$
n\sum_{k\neq 0}\left(   sinc^{(n-1)}\left(-k\right)  + sinc^{(n)}\left(-k\right)
   \right)\frac{e^{kH}{\bf a}-{\bf a}}{k}.
\end{equation}
\end{thm}

\begin{rem}
We note that $(sinc\>t)^{(m)}(0)=(-1)^{m}/(m+1)!$ if $m$ is even, and  $(sinc\>t)^{(m)}(0)=0$ if $m$ is odd.
\end{rem}
\begin{proof}
For any ${\bf a},\>{\bf a}^{*}\in l^{2}$ the function $\Phi(t)=\left<e^{tH}{\bf a},\>{\bf a}^{*}\right>$ belongs to $B_{\pi}^{\infty}(\mathbb{R})$.
We consider $\Psi$ which was introduced previously in (\ref{F1}), (\ref{F2}). 
We have
$$
\Psi(t)=\sum_{k\in \mathbb{Z}}\Psi\left(k\right)\>  sinc\left(t-k\right),
$$
where the series converges in $l^{2}$. 
From here we obtain the next formula
$$
\left(\frac{d}{dt}\right)^{n}\Psi(t)=\sum_{k\in \mathbb{Z}}\Psi \left(k\right)  sinc^{(n)}\left(t-k\right)
$$
and  since 
$$
\left(\frac{d}{dt}\right)^{n}\Phi(t)=n\left(\frac{d}{dt}\right)^{n-1}\Psi(t)+t\left(\frac{d}{dt}\right)^{n}\Psi(t)
$$
we obtain 
$$
\left(\frac{d}{dt}\right)^{n}\Phi(t)=n\sum_{k\in \mathbb{Z}}\Psi\left(k\right)\>  sinc^{(n-1)}\left(t-k\right)+t\>\sum_{k\in \mathbb{Z}}\Psi\left(k\right)\> sinc^{(n)}\left(t-k\right).
$$
Since
$
\left(\frac{d}{dt}\right)^{n}\Phi(t)=\left<H^{n}e^{tH}{\bf a}, {\bf a}^{*}\right>,
$
and
$$
\Psi\left(k\right)=\left<\frac{e^{kH}{\bf a}-{\bf a}}{k}, {\bf a}^{*}\right>,
$$
we obtain that
the formulas (\ref{s2})-(\ref{s3})  hold. 
Theorem is proved. 
\end{proof}

\section{The case of the Kak-Hilbert transform}\label{KH}

We also briefly show how our methods can be applied to the Kak-Hilbert transform to obtain similar sampling and interpolation formulas. Kak-Hilbert transform also generates a one parameter group of operators in $l^{2}$ but it is not a group of isometries 
like in the case of the discrete Hilbert transform. However, this group of operators is uniformly bounded. This uniform boundness 
is explored  to include the case of Kak-Hilbert transform into our scheme.

The Kak-Hilbert transform $$K{\bf a}= {\bf b},\>\>{\bf a}=(a_{n})\in l^{2}, \>\>{\bf b}=(b_{n})\in l^{2},$$ is defined by the formula
$$
b_{m}=\frac{2}{\pi}\sum_{n\>even}\frac{a_{n}}{m-n},
$$
if $m$ is odd, and by the formula
$$
b_{m}=\frac{2}{\pi}\sum_{n\>odd}\frac{a_{n}}{m-n},
$$
if $m$ is even.

It is known that $K$ is an isometry in $l^{2}$ (see \cite{Kak}). As a bounded operator $K$ generates a one-parameter group $e^{t K}$ of bounded operators in $l^{2}$. One can verify the property 
$K^{2}=-I$ which implies the
explicit formula for $e^{t K}$ (see \cite{DC}):
$$
e^{t K}=(\cos t)\>I+(\sin t)\>K,
$$
which gives the uniform bound $\|e^{t K}\|\leq 2$.

Pick an ${\bf a}^{*}\in l^{2}$ and  consider a scalar-valued function 
$$
F(t)=\langle e^{t K}{\bf a}, {\bf a}^{*}\rangle,\>\>\>\>t\in \mathbb{R}.
$$
Note, that since $K$ is an isometry, the analog of the Bernstein inequality takes the form 
$$
\|K^{n}{\bf a}\|=\|{\bf a}\|,\>\>\>n\in \mathbb{N},
$$
(compare to (\ref{bern-1}).
Using this inequality one can easily 
 prove the following analog of the Lemma \ref{LEM}.

\begin{lem}\label{LEM-1}
For every ${\bf a}\in l^{2}$ and every ${\bf a}^{*}\in l^{2}$ the function $F$ belongs to the Bernstein class ${\bf B}_{1}^{\infty}(\mathbb{R}).$ 

\end{lem}

We also have the following Corollary similar to (\ref{COL}). 

\begin{col}
For a fixed ${\bf a}\in l^{2}$ the vector-valued function
$$
e^{t K}{\bf a}: \mathbb{R}\mapsto l^{2},
$$
has extension $ e^{z  K}{\bf a},\>\>z\in \mathbb{C},$ to the complex plain as an entire function of the exponential type $1$ which is bounded on the real line.
\end{col}

Similarly to the case of the discrete Hilbert transform one could prove the following statements.

\begin{thm}For  every ${\bf a},\>{\bf a}^{*}\in l^{2}$ and every $0<\gamma< 1$ the following formulas hold true.

$$
\langle e^{t K}{\bf a}, {\bf a}^{*}\rangle=\sum_{n\in \mathbb{Z}} \langle e^{(\gamma n \pi)K}{\bf a}, {\bf a}^{*}\rangle\>
sinc\left(\frac{t}{\gamma \pi}-n\right),
$$
where the series converges uniformly on compact subsets of $\mathbb{R}$.

The following formulas also hold true
$$
e^{t K}{\bf a}=
 {\bf a}+t\>sinc \>\left(t/\pi\right)  K{\bf  a}\>+
t\sum_{n\in \mathbb{Z}\setminus \{0\}}  \frac{e^{n\pi K}{\bf a}-{\bf a}}{n\pi}
\>
sinc\left(\frac{t }{\pi}-n\right),
$$
where the series converges in the norm of $l^{2}$,
and
$$
e^{t K}{\bf a}=sinc\left( \frac{t}{\pi}\right){\bf a}+t \>sinc\left( \frac{t}{\pi}\right)K{\bf a}+  \sum_{n\in \mathbb{Z}\setminus \{0\}}\frac{ t}{n\pi}sinc\left( \frac{t}{\pi}-n\right)e^{n\pi K}{\bf a},
$$
where the series converges in $l^{2}$.
\end{thm}
One could also reformulate for the Kak-Hilbert transform all other results which were obtained for the discrete Hilbert transform. 
In particular, one could  introduce bounded operators
$$
\mathcal{T}^{(2m-1)}_{K}{\bf a}=
\sum_{n\in \mathbb{Z}}(-1)^{n+1}A_{m,n}
e^{(n-1/2)\pi K}{\bf a}
,\>\> \>\>{\bf a}\in l^{2},\>\>\>m\in \mathbb{N},
$$
and
$$
\mathcal{T}^{(2m)}_{K}{\bf a}=
\sum_{n\in \mathbb{Z}}(-1)^{n+1}B_{m,n}
e^{n\pi K}
{\bf a}, \>\> \>\>{\bf a}\in l^{2}, \>\>m\in \mathbb{N},
$$
and to prove relevant Riesz-Boas-type interpolation formulas

$$
K^{r}{\bf a}=\mathcal{T}^{(r)}_{K}{\bf a},\>\>\>{\bf a}\in l^{2},\>\>\>r\in \mathbb{N}.
$$

\bigskip

{\bf Acknowledgements.}  I would like to thank  Professor Carlo Bardaro and Professor Rudolf Stens for interesting and constructive discussions and suggestions.

\end{document}